\documentclass[11pt]{amsart}
\usepackage{amsfonts,amssymb,amsthm,eucal,color}

\usepackage[margin=1.25in]{geometry}

\newtheorem{thm}{Theorem}
\newtheorem{cor}[thm]{Corollary}
\newtheorem{lemma}[thm]{Lemma}
\newtheorem{prop}[thm]{Proposition}
\newtheorem{defn}[thm]{Definition}

\newcommand{\R}{\mathbb{R}}
\newcommand{\F}{\mathbb{F}}

\newcommand{\C}{\mathbb{C}}

\newcommand{\inprod}[2]{\left\langle #1, #2 \right\rangle}
\renewcommand{\Re}{\operatorname{Re}}

\newcommand{\abs}[1]{\left\vert #1 \right\vert}
\newcommand{\norm}[1]{\left\Vert #1 \right\Vert}

\newcommand{\eps}{\varepsilon}
\newcommand{\dfn}[1]{{\bf #1}}
\newcommand{\Set}[2]{\left\{#1 \mathrel{} \middle| \mathrel{} #2 \right\}}

\DeclareMathOperator{\vol}{vol}

\newcommand{\udivdim}{\overline{\dim}_{\mathrm{Div}}}
\newcommand{\umagdim}{\overline{\dim}_{\mathrm{Mag}}}
\newcommand{\uboxdim}{\overline{\dim}_{\mathrm{Mink}}}
\newcommand{\ldivdim}{\underline{\dim}_{\mathrm{Div}}}
\newcommand{\lmagdim}{\underline{\dim}_{\mathrm{Mag}}}
\newcommand{\lboxdim}{\underline{\dim}_{\mathrm{Mink}}}
\newcommand{\divdim}{\dim_{\mathrm{Div}}}
\newcommand{\magdim}{\dim_{\mathrm{Mag}}}
\newcommand{\boxdim}{\dim_{\mathrm{Mink}}}
\DeclareMathOperator{\diam}{diam}

\numberwithin{thm}{section}
\numberwithin{equation}{section}

\allowdisplaybreaks[3]

\author{Mark W.\ Meckes}

\title[Magnitude, \&c.]{Magnitude, diversity, capacities, and
  dimensions of metric spaces}

\address{Case Western Reserve University, Cleveland, OH 44106, U.S.A.
Phone: +01-216-368-4997.  Fax: +01-216-368-5163.}
\email{mark.meckes@case.edu}



\begin{document}

\begin{abstract}
  Magnitude is a numerical invariant of metric spaces introduced by
  Leinster, motivated by considerations from category theory.  This
  paper extends the original definition for finite spaces to compact
  spaces, in an equivalent but more natural and direct manner than in
  previous works by Leinster, Willerton, and the author.  The new
  definition uncovers a previously unknown relationship between
  magnitude and capacities of sets.  Exploiting this relationship, it
  is shown that for a compact subset of Euclidean space, the magnitude
  dimension considered by Leinster and Willerton is equal to the
  Minkowski dimension.
\end{abstract}

\maketitle


\section{Introduction}

The magnitude of a metric space is a numerical isometric invariant
introduced by Leinster in \cite{Leinster}.  From the perspective of
geometry, its definition was motivated in a rather unusual way.  In
\cite{Leinster-Euler}, Leinster had defined the Euler characteristic
of a finite category, which generalizes the Euler characteristic of a
topological space or of a poset.  This notion of Euler characteristic
can be naturally generalized from categories to enriched
categories, a family of algebraic structures which, as observed by
Lawvere in \cite{Lawvere}, includes metric spaces; in this context the
generalization of Euler characteristic is named ``magnitude''.
Specialized then to metric spaces, one obtains Leinster's definition
of the magnitude of a finite metric space, stated in Definition
\ref{D:finite-magnitude} below.  Magnitude was extended to compact
metric spaces in multiple ways in
\cite{Leinster,LW,Willerton-heuristic,Willerton-homog}, which were
shown by the author in \cite{Meckes} to agree with each other for
many spaces (specifically, for so-called positive definite spaces,
which include all compact subsets of Euclidean space).

Given this exotic provenance, it may come as a surprise that magnitude
turns out to be closely related to classical invariants of integral
geometry; see \cite{Leinster,LW,Willerton-homog} for a number of
results along these lines.  Conjectures in \cite{Leinster,LW}, which
are supported by partial results in those papers and by heuristics and
numerical computations in \cite{Willerton-heuristic}, suggest that the
relationship between magnitude and integral geometry runs deeply
enough that all the intrinsic volumes of convex bodies can be
recovered from magnitude.  A second surprise, in a completely
different direction, is that the magnitude of a finite metric space
has been introduced in the literature before, in connection with
quantifying biodiversity in \cite{SoPo}.  Although the theory of
magnitude was not developed at all in \cite{SoPo}, the relationship
between magnitude and diversity has been investigated more fully in
\cite{Leinster-entropy}.

The present work grew out of the author's search for a more
satisfactory definition of magnitude for compact metric spaces.  The
approach of \cite{Meckes} was to introduce yet another definition, a
measure-theoretic generalization of a variational formula for the
magnitude of a finite positive definite space derived in
\cite{Leinster-entropy,Leinster}, and to prove that this new
definition agrees with all the earlier ones.  Here we instead take a
more functional-analytic approach to generalize directly the original
definition of magnitude for finite metric spaces.  Besides the
{\ae}sthetic appeal of a more direct approach, the resulting
definition, which again agrees with the earlier ones, can be used to
prove new properties of magnitude in Euclidean space.  More
interestingly, it uncovers previously unknown connections between
magnitude and potential theory.  In fact, another surprise is that the
magnitude of a compact subset of Euclidean space has (almost) been
introduced in an equivalent form in the literature before both
\cite{Leinster} and \cite{SoPo}, as a type of capacity.  (It seems
likely that magnitude is the unique notion to have arisen
independently in potential theory, theoretical ecology, and category
theory.)

The relationship between magnitude and capacity has important
consequences.  There is a notion of dimension associated to magnitude
which was first investigated in \cite{LW}, and which provided some of
the first compelling evidence that magnitude encodes interesting
geometric information.  Using in part a deep result on relationships
between different capacities, we will see that in Euclidean space,
this magnitude dimension turns out to be the same as Minkowski
dimension.  In establishing this result, we find an apparently new
formulation of Minkowski dimension in terms of capacity.  In addition,
the conjectures from \cite{Leinster,LW} mentioned earlier would, if
true, indicate previously unknown connections between capacities and
intrinsic volumes of convex bodies.

\medskip

The rest of this paper is organized as follows.  Section
\ref{S:finite-spaces} presents Leinster's original definition of the
magnitude of a finite metric space and some related definitions we
will need.  Section \ref{S:weightings} develops the
functional-analytic generalization of the definition of magnitude for
compact spaces.  Section \ref{S:potential} presents a dual perspective
on the definitions of Section \ref{S:weightings}, and discusses a
quantity closely related to magnitude, the maximum diversity of a
metric space.  Section \ref{S:Euclidean} specializes the constructions
of the previous sections to subsets of Euclidean space, and uses them
to prove new results about the behavior of magnitude in Euclidean
space.  Section \ref{S:capacity} discusses the connections between
magnitude, maximum diversity, and capacities.  Section
\ref{S:dimension} proves a new characterization of Minkowski dimension
in terms of maximum diversity, and uses a result from potential theory
recalled in Section \ref{S:capacity} to deduce that magnitude
dimension and Minkowski dimension are equal in Euclidean space.
Finally, in Section \ref{S:ultrametric}, we briefly investigate
another, closely related instance of the magnitude of an enriched
category: the case of ultrametric spaces, whose theory turns out to be
much simpler but nevertheless intriguingly similar to that of metric
spaces.

\subsection*{Acknowledgements}

The author thanks Pierre Albin, Shiri Artstein-Avidan, Harsh Mathur,
and especially Tom Leinster for helpful discussions.  Particular
thanks are due to Charles Clum for convincing the author of the
usefulness of the potential function in analyzing magnitude.  This
work was supported in part by NSF grant DMS-0902203.


\section{Finite metric spaces}
\label{S:finite-spaces}

We now recall the definition of the magnitude of a finite metric space.

\begin{defn} \label{D:finite-magnitude}
Given a finite metric space $(A, d)$, define the matrix $\zeta
\in \R^{A \times A}$ by $\zeta(a,b) := e^{-d(a,b)}$. A vector $w \in
\R^A$ is a \dfn{weighting} for $A$ if for each $a \in A$,
\begin{equation} \label{E:zeta-w} 
  (\zeta w)(a) = \sum_{b \in A} e^{-d(a,b)} w(b) = 1.
\end{equation}
If $A$ possesses a weighting $w$, then the \dfn{magnitude} of $A$ is
\begin{equation} \label{E:finite-magnitude} 
  \abs{A} := \sum_{a \in A} w(a).
\end{equation}
\end{defn}

An arbitrary metric space need not possess a weighting (see
\cite[Example 2.2.7]{Leinster}), nor must a weighting be unique.
However, it is easy to check that if a weighting exists, then the
value of the sum in \eqref{E:finite-magnitude} is independent of the
weighting.  Thus this definition of magnitude can only be defined on a
proper subclass of the class of finite metric spaces, but magnitude is
well-defined on that subclass.  The magnitude of $A$ is always defined
if the matrix $\zeta$ is invertible, which is the case whenever $A$ is
a subset of Euclidean space \cite[Theorem 2.5.3]{Leinster}.

There is an arbitrary choice of scale implicit in Definition
\ref{D:finite-magnitude}: instead of the metric space $A=(A,d)$, one
may equally well consider any of the metric spaces $tA = (A,td)$ for
$t > 0$.  We thus define the \dfn{magnitude function} of $A$ to be the
function $t \mapsto \abs{tA}$ for $t > 0$.  In general the magnitude
function may be only partially defined, but it is always defined for
all but finitely many values of $t$ \cite[Proposition
2.2.6]{Leinster}.

Observe that if $A \subseteq \R^n$, then the metric space $tA$ as
defined above is isometric to the space $\Set{ta}{a \in A} \subseteq
\R^n$; as is standard, we will thus use $tA$ to denote this latter set
without fear of ambiguity.  In particular, the magnitude function of a
finite subset of Euclidean space is defined everywhere.

If we write $\zeta_t$ for the matrix associated to $tA$ by Definition
\ref{D:finite-magnitude}, so that $\zeta_t(a,b) = e^{-t d(a,b)}$, then
$\zeta_t$ tends toward the identity matrix indexed by $A$ when $t \to
\infty$.  From this it follows that for sufficiently large $t$, $tA$
has a unique weighting which tends to the vector whose entries are all
$1$, and the magnitude of $tA$ tends to the cardinality of $A$.  (See
\cite[Proposition 2.2.6]{Leinster} for more details.) This suggests
the intuitive interpretation of magnitude as the ``effective number of
points'' of $A$ when $A$ is viewed at a particular scale.  (Indeed, in
\cite{SoPo}, the magnitude of $A$ is called the ``effective number of
species'' in an ecosystem whose species are represented by the points
of $A$, when $e^{-d(a,b)}$ is interpreted as the similarity between
species $a$ and species $b$.)


\section{Weightings and magnitude}
\label{S:weightings}

In this and the next section, it will make no difference whether we
choose to work with real or complex scalars.  However, in section
\ref{S:Euclidean}, Fourier-analytic tools will be brought to bear and
it will be more natural to work with complex scalars.  For now we
write $\F$ for the scalar field, which may be taken as either $\R$ or
$\C$.

It is useful to introduce an ambient, possibly noncompact metric space
$(X, d)$ containing a compact space $A$ whose magnitude we wish to
consider, particularly in order to apply Fourier analysis when $A$ is
a subset of Euclidean space. Denote by $FM(X)$ the space of finitely
supported, finite signed (if $\F = \R$) or complex (if $\F = \C$)
measures on $X$.  Define a symmetric bilinear (if $\F = \R$) or
Hermitian sesquilinear (if $\F = \C$) form
$\inprod{\cdot}{\cdot}_\mathcal{W}$ on $FM = FM(X)$ by
\[
\inprod{\mu}{\nu}_\mathcal{W} := \int \int e^{-d(a,b)} \ d\mu(a) \
d\overline{\nu}(b),
\]
where $\overline{\nu}$ denotes the complex conjugate of a complex
measure $\nu$.  The metric space $(X, d)$ is said to be \dfn{positive
  definite} if $\inprod{\cdot}{\cdot}_\mathcal{W}$ is a positive
definite inner product. Equivalently, $X$ is positive definite if for
each nonempty finite subset $A \subseteq X$, the associated symmetric
matrix $\zeta$ from Definition \ref{D:finite-magnitude} is positive
definite.  The phrase ``positive definite metric space'' will be
abbreviated as \dfn{PDMS}.  In particular, $\R^n$ is positive definite
\cite[Theorem 2.5.3]{Leinster}.

In order to talk sensibly of the magnitude of $tA$ for each $t > 0$,
we will sometimes need the additional assumption that $tA$ is a PDMS
for each $t$. By \cite[Theorem 3.3]{Meckes}, this property of $A$ is
equivalent to the classical property of \dfn{negative type} (whose
original definition will not be needed here).  Every subset $A
\subseteq \R^n$ is of negative type; many other spaces of negative
type which are of interest are collected in \cite[Theorem
3.6]{Meckes}.

For the rest of this paper $(X, d)$ is assumed to be a PDMS.  Let
$\mathcal{W}$ denote the completion of $FM$ with respect to the inner
product $\inprod{\cdot}{\cdot}_\mathcal{W}$; we call $\mathcal{W}$ the
\dfn{weighting space} of $X$.  For a compact subset $A \subseteq X$,
denote by $\mathcal{W}_A$ the closure in $\mathcal{W}$ of $FM(A)$.
Note that $\mathcal{W}_A$ is simply the weighting space of $A$, and is
independent of the ambient space $X$.  (Another characterization of
$\mathcal{W}$, which may feel more concrete to some readers, will be
given in Section \ref{S:potential} below.)

Denote by $C^{1/2}$ the space of H\"older continuous functions $f:X
\to \F$ with exponent $1/2$, equipped with the norm
\[
\norm{f}_{C^{1/2}} := \max\left\{ \norm{f}_\infty,
\sup_{x,y \in X, \ x\neq y} \frac{\abs{f(x) -
    f(y)}}{\sqrt{2 d(x,y)}}\right\}.
\]
We adopt this slightly unusual version of the $C^{1/2}$ norm purely
for convenience; any equivalent norm would work equally well for our
purposes.

For $\mu \in FM$, define $Z\mu:X \to \F$ by
\begin{equation}\label{E:Z}
  Z\mu(x) := \int e^{-d(x,y)} \ d\mu(y),
\end{equation}
so that
\begin{equation}\label{E:FM-FM}
  \inprod{\mu}{\nu}_{\mathcal{W}} = \int (Z \mu) \ d\overline{\nu}
  = \inprod{Z\mu}{\overline{\nu}}
\end{equation}
for each $\mu, \nu \in FM$, where $\inprod{\cdot}{\cdot}$ denotes the
standard bilinear pairing between functions and measures.

\begin{lemma}\label{T:FM-bounded}
  If $\mu \in FM$, then $Z\mu \in C^{1/2}$ and $\norm{Z\mu}_{C^{1/2}} \le
  \norm{\mu}_{\mathcal{W}}$.
\end{lemma}

\begin{proof}
  By the Cauchy--Schwarz inequality, for each $x, y \in X$,
  \[
  \abs{Z\mu(x)} = \abs{\inprod{\mu}{\delta_x}_{\mathcal{W}}}
  \le \norm{\mu}_{\mathcal{W}} \norm{\delta_x}_{\mathcal{W}} = \norm{\mu}_{\mathcal{W}}
  \]
  and
  \[
  \abs{Z\mu(x) - Z\mu(y)} = \abs{\inprod{\mu}{\delta_x - \delta_y}_{\mathcal{W}}}
  \le \norm{\mu}_{\mathcal{W}} \norm{\delta_x - \delta_y}_{\mathcal{W}}.
  \]
  Now
  \begin{align*}
  \norm{\delta_x - \delta_y}_{\mathcal{W}}^2 &= \norm{\delta_x}_{\mathcal{W}}^2
  - \inprod{\delta_x}{\delta_y}_{\mathcal{W}}
  - \inprod{\delta_y}{\delta_x}_{\mathcal{W}}
  + \norm{\delta_y}^2\\
  &= 2\bigl( 1 - e^{-d(x,y)} \bigr) \le 2 d(x,y),
  \end{align*}
  so $\norm{Z\mu}_{C^{1/2}} \le \norm{\mu}_{\mathcal{W}}$.
\end{proof}

\begin{prop}\label{T:W-bounded}
  The map $Z:FM \to C^{1/2}$ defined by \eqref{E:Z} extends uniquely
  to an injective linear operator $Z: \mathcal{W} \to C^{1/2}$ with
  $\norm{Z} = 1$.  Furthermore, for each $w \in \mathcal{W}$ and
  $\mu \in FM$,
  \begin{equation}\label{E:FM-W}
  \inprod{w}{\mu}_{\mathcal{W}} = \int (Z w) \ d\overline{\mu}.
  \end{equation}
\end{prop}

\begin{proof}
  Lemma \ref{T:FM-bounded} imply that $Z$ extends uniquely to a linear
  operator $Z:\mathcal{W} \to C^{1/2}$ with norm at most $1$, and
  \eqref{E:FM-FM} then implies \eqref{E:FM-W}. For each $x \in X$,
  $\norm{Z \delta_x}_\infty = 1 = \norm{\delta_x}_{\mathcal{W}}$, so
  $\norm{Z} = 1$.

  To prove injectivity, suppose that $Z w = 0$ for some $w \in
  \mathcal{W}$.  Then for each $x \in X$, by \eqref{E:FM-W},
  \[
  0 = Z w(x) = \inprod{w}{\delta_x}_{\mathcal{W}},
  \]
  from which it follows by linearity that
  $\inprod{w}{\mu}_{\mathcal{W}} = 0$ for each $\mu \in FM$.  Since
  $FM$ is dense in $\mathcal{W}$, this implies that $w = 0$.
\end{proof}

\begin{defn} \label{D:compact-magnitude}
  Let $A \subseteq X$ be compact.  A \dfn{weighting} for $A$ is an
  element $w \in \mathcal{W}_A$ such that for each $a \in A$, $Zw(a) =
  1$.  If $A$ possesses a weighting $w$, then the \dfn{magnitude} of $A$
  is
  \begin{equation} \label{E:compact-magnitude}
  \abs{A} := \norm{w}_{\mathcal{W}}^2.
  \end{equation}
  If $A$ does not possess a weighting, then $\abs{A} := \infty$.
\end{defn}

Since $Z$ is injective, if $A$ possesses a weighting, then the
weighting is unique.  If $A$ possesses a complex weighting $w$, then
the real part of $w$ (defined by extending the ``real part'' map on
$FM$ to $\mathcal{W}$) is also a weighting for $A$, and hence equal to
$w$; thus the existence of a weighting for $A$ and the magnitude of
$A$ are independent of the scalar field.  It is an open question
whether there exists a compact PDMS whose magnitude is infinite.

If $A$ is a compact metric space of negative type, then \dfn{magnitude
  function} of $A$ is defined as before to be the function $t \mapsto
\abs{tA}$ for $t > 0$.

We next compare Definition \ref{D:compact-magnitude} to the original
Definition \ref{D:finite-magnitude} of magnitude for a finite metric
space $A$.  The definitions of weighting are clearly equivalent when
$w \in \R^A$ is identified with $\sum_{a\in A} w(a) \delta_a \in
FM(A)$. As noted earlier, an arbitrary finite metric space may not
possess a weighting, but if $A$ is a finite PDMS then the matrix
$\zeta$ is positive definite, hence invertible, so $A$ possesses a
unique weighting $w \in \R^A$.

The equivalence of \eqref{E:finite-magnitude} and
\eqref{E:compact-magnitude} is less immediately obvious, so it is
desirable to motivate \eqref{E:compact-magnitude}.  If one thinks of
measures as dual to functions, then the right hand side of
\eqref{E:finite-magnitude} is interpreted as
$\inprod{1}{\overline{w}}$, where $1$ denotes the function with the
constant value $1$.  However, interesting function spaces on
noncompact domains generally do not contain constant functions, which
suggests that this interpretation is also unsatisfactory; one should
replace the constant $1$ with some canonical function which is equal
to $1$ everywhere on $A$. Fortunately, \eqref{E:zeta-w} provides such
a function, namely $\zeta w$. The definition
\eqref{E:finite-magnitude} can thus be rewritten as
\[
\abs{A} := \sum_{a, b \in A} w(a) \zeta(a,b) \overline{w(b)}.
\]
and then reinterpreted as
\[
\abs{A} := \int (\zeta w) \ d\overline{w} 
= \inprod{\zeta w}{\overline{w}},
\]
which is clearly equivalent to \eqref{E:compact-magnitude}.

The next result shows that Definition \ref{D:compact-magnitude} agrees
with the definition of magnitude adopted in \cite{Meckes} (cf.\
Theorem 2.4 in \cite{Meckes}), and motivates the definition of
$\abs{A} = \infty$ when $A$ does not possess a weighting.

\begin{thm} \label{T:agreement} If $A \subseteq X$ is compact then
  \begin{equation}\label{E:magnitude-sup}
    \abs{A} = \sup \Set{\frac{\abs{\mu(A)}^2}{\norm{\mu}_{\mathcal{W}}^2}}
    {\mu \in FM(A), \ \mu \neq 0}.
  \end{equation}
\end{thm}

\begin{proof}
  Let $\kappa$ denote the supremum in \eqref{E:magnitude-sup}.  We
  need to show that $A$ possesses a weighting $w$ if and only if
  $\kappa < \infty$, and that in that case $\norm{w}_{\mathcal{W}}^2 =
  \kappa$.
  
  Suppose that $A$ possesses a weighting $w$.  Then for each $\mu \in
  FM(A)$, by \eqref{E:FM-W},
  \[
  \inprod{w}{\mu}_{\mathcal{W}} = \inprod{Zw}{\overline{\mu}} =
  \overline{\mu(A)},
  \]
  and so by the Cauchy--Schwarz inequality,
  \[
  \abs{\mu(A)}^2 = \abs{\inprod{w}{\mu}_{\mathcal{W}}}^2 \le
  \norm{\mu}_{\mathcal{W}}^2 \norm{w}_{\mathcal{W}}^2,
  \]
  which implies that $\kappa < \infty$.

  Now suppose that $\kappa < \infty$.  Then the linear functional $\mu
  \mapsto \mu(A)$ on $\bigl(FM(A), \norm{\cdot}_{\mathcal{W}}\bigr)$
  is bounded with norm $\sqrt{\kappa}$. It thus extends to a linear
  functional on $\mathcal{W}_A$ with norm $\sqrt{\kappa}$. By the
  Riesz representation theorem for bounded linear functionals on
  Hilbert spaces, there exists a $w \in \mathcal{W}_A$ such that
  $\inprod{\mu}{w}_{\mathcal{W}} = \mu(A)$ for each $\mu \in FM(A)$
  and $\norm{w}_{\mathcal{W}}^2 = \kappa$.  In particular, for each $a
  \in A$, by \eqref{E:FM-W},
  \[
  Zw(a) = \inprod{w}{\delta_a}_{\mathcal{W}}
  = \overline{\delta_a(A)} = 1.
  \]
  Thus $w$ is a weighting for $A$, and $\abs{A} =
  \norm{w}_{\mathcal{W}}^2 = \kappa$.
\end{proof}

Besides the agreement of Definition \ref{D:compact-magnitude} with the
original definition of magnitude for finite PDMSs, the results of
\cite[Section 2]{Meckes} support Definition \ref{D:compact-magnitude}
as the ``correct'' notion of magnitude for a compact PDMS.  As shown
in \cite[Theorem 2.6]{Meckes}, magnitude as defined here or in
\cite{Meckes} is lower semicontinuous on the class of compact PDMSs
equipped with the Gromov--Hausdorff topology. In fact, since
Theorem \ref{T:agreement} implies that
\[
\abs{A} = \sup \Set{\abs{B}}{B \subseteq A \text{ is finite}},
\]
magnitude as defined here is the maximal lower semicontinuous
extension of magnitude from the class of finite PDMSs to the class of
compact PDMSs.


\section{The reproducing kernel and potential function of a PDMS}
\label{S:potential}

For technical reasons, it is fruitful to shift the emphasis from the
weighting $w$ of a compact subset $A \subseteq X$ to the function $Zw:
X \to \F$, and thus from the weighting space $\mathcal{W}$ to the
function space $Z(\mathcal{W}) \subseteq C^{1/2}$, which we consider
next.  The results of the next two sections should convince the
skeptical reader that this additional layer of complexity is
worthwhile.

Define $\mathcal{H} := Z(\mathcal{W}) \subseteq C^{1/2}$, and for a
compact set $A \subseteq X$, define $\mathcal{H}_A :=
Z(\mathcal{W_A})$.  Equip $\mathcal{H}$ with the inner product
\[
\inprod{g}{h}_{\mathcal{H}} := \inprod{Z^{-1}g}{Z^{-1}h}_{\mathcal{W}},
\]
recalling that $Z$ is injective by Proposition \ref{T:W-bounded}.
Then $Z:\mathcal{W} \to \mathcal{H}$ is a surjective isometry, and
$\norm{h}_{C^{1/2}} \le \norm{h}_{\mathcal{H}}$ for each $h \in
\mathcal{H}$.  Furthermore, the Hilbert spaces $\mathcal{H}$ and
$\mathcal{W}$ act as duals to each other via the bilinear pairing
\begin{equation} \label{E:H-W}
\inprod{h}{w} = \inprod{Z^{-1} h}{\overline{w}}_\mathcal{W} = \inprod{h}{\overline{Z
  w}}_\mathcal{H}.
\end{equation}

By \eqref{E:FM-W}, \eqref{E:H-W} extends the standard pairing between
functions and measures.  In particular, for each $x \in X$,
\[
h(x) = \inprod{h}{\delta_x} = \inprod{h}{e^{-d(x, \cdot)}}_{\mathcal{H}}.
\]
That is, $\mathcal{H}$ is the reproducing kernel Hilbert space (RKHS)
on $X$ with the reproducing kernel $e^{-d(x,y)}$. (Readers unfamiliar
with RKHSs are referred to \cite{Aronszajn}.) We have chosen to define
$\mathcal{H}$ in terms of $\mathcal{W}$ since this more closely
parallels Leinster's original definition of the magnitude of a finite
metric space.  Alternatively, one could first define $\mathcal{H}$ to
be this RKHS, and then define $\mathcal{W}$ as the dual space of
$\mathcal{H}$. We will not make use of the theory of RKHSs in this
paper, opting instead for more self-contained arguments.

If a compact set $A \subseteq X$ possesses a weighting $w$, then the
function $h = Z w \in C^{1/2}$ is called the \dfn{potential function}
of $A$.  This name is motivated by the following (entirely
non-realistic) physical model, which amounts to a less picturesque
version of the penguin analogy discussed in
\cite{Willerton-heuristic}.  (Actually, it is related to the Yukawa
potential, but only in the one-dimensional case; see \cite[Section
6.23]{LiLo}.)

We posit a type of charge such that a unit charge at $x \in X$ creates
a potential at $y \in X$ of $e^{-d(x,y)}$.  A finite set of fixed
points $A \subseteq X$ represents a conductor which is connected to a
ground to and from which charge may flow freely, but which otherwise
does not interact with anything in the space $X$.  The conductor is
held at a uniform potential of $1$, and no other charge exists
anywhere in $X$.  Then the charge at each point of $A$ is given by the
weighting $w$ of $A$, and
\[
h(x) = \sum_{a\in A} e^{-d(a,x)} w(a)
\]
is indeed the potential at each point $x \in X$.  The magnitude
$\abs{A}$ is then the total charge on $A$. If $A \subseteq X$ is an
infinite compact set, its weighting $w$ represents a charge
distribution, although we stress that $w$ need not be given by a
function or even a measure on $A$; in general it will be some more
singular type of object.  When $X$ is Euclidean space, $w$ will in
fact turn out to be a distribution in the sense of Schwartz, as seen
in the next section.

We can now give a finiteness condition and a variational formula for
magnitude, dual to those in Theorem \ref{T:agreement}, in terms of
the RKHS $\mathcal{H}$.

\begin{thm} \label{T:inf-magnitude} Let $A \subseteq X$ be compact.
  Then $A$ possesses a weighting if and only if there exists a
  function $h \in \mathcal{H}$ such that $h \equiv 1$ on $A$.  In that
  case,
  \begin{equation}\label{E:inf-magnitude}
  \abs{A} = \inf \Set{ \norm{h}_{\mathcal{H}}^2}{h \in
    \mathcal{H} \text{ and } h \equiv 1 \text{ on } A},
  \end{equation}
  and the infimum is uniquely attained by the potential function of $A$.
\end{thm}

\begin{proof}
  If $A$ possesses a weighting $w$, then by definition the potential
  function $h = Zw$ of $A$ lies in the set appearing in the right hand
  side of \eqref{E:inf-magnitude}.  Moreover, $\abs{A} =
  \norm{w}_{\mathcal{W}}^2 = \norm{h}_{\mathcal{H}}^2$.

  On the other hand, if $h \in \mathcal{H}$ and $h \equiv 1$ on $A$,
  then for each $\mu \in FM(A)$,
  \[
  \abs{\mu(A)}^2 = \abs{\inprod{h}{\mu}}^2 
  \le \norm{h}_{\mathcal{H}}^2 \norm{\mu}_{\mathcal{W}}^2.
  \]
  By Theorem \ref{T:agreement}, $A$ possesses a weighting and
  $\abs{A} \le \norm{h}_{\mathcal{H}}^2$.  It furthermore follows that
  if $A$ possesses a weighting then \eqref{E:inf-magnitude} holds.

  Moreover, if the closed affine subspace
  \[
  \Set{h \in \mathcal{H}}{h \equiv 1 \text{ on } A}
  = \bigcap_{a \in A} \Set{h}{\inprod{h}{\delta_a} = 1}
  \]
  is nonempty then it contains a unique element of minimal norm, which
  must therefore be the potential function $Zw$ of $A$.
\end{proof}

As mentioned in Section \ref{S:weightings}, it is an open question
whether an arbitrary compact PDMS $A$ has finite magnitude.  By
Theorem \ref{T:inf-magnitude}, this is equivalent to asking whether
the RKHS $\mathcal{H}_A$ contains the constant functions on $A$.

The next two results are useful in computing magnitudes, as will be
seen in Section \ref{S:Euclidean}.

\begin{prop} \label{T:inprod-magnitude} Let $A \subseteq X$ be compact
  with weighting $w$.  Then \( \abs{A} = \inprod{h}{w} \) for any $h
  \in \mathcal{H}$ such that $h \equiv 1$ on $A$.
\end{prop}

\begin{proof}
  For any such $h$ and any $\mu \in FM(A)$,
  \[
  \inprod{h}{\mu} = \mu(A) = \inprod{\overline{Zw}}{\mu} =
  \inprod{\mu}{w}_{\mathcal{W}}.
  \]
  Let $\mu_n$ be a sequence in $FM(A)$ which converges to $w$ with
  respect to $\norm{\cdot}_{\mathcal{W}}$.  Then
  \[
  \inprod{h}{w} = \lim_{n\to \infty} \inprod{h}{\mu_n}
  = \lim_{n\to  \infty} \inprod{\mu_n}{w}_{\mathcal{W}} =
  \inprod{w}{w}_{\mathcal{W}} = \abs{A}.
  \]

  Alternatively, the proposition is equivalent to the claim that
  $\abs{A} = \inprod{h}{Zw}_{\mathcal{H}}$ for any $h \in \mathcal{H}$
  such that $h \equiv 1$ on $A$, where $Zw$ is the potential
  function of $A$.  In this form the statement follows from
  Theorem \ref{T:inf-magnitude} and basic Hilbert space geometry.
\end{proof}

\begin{cor} \label{T:potential-orthogonal} Let $A \subseteq X$ be
  compact with potential function $h$. Then
  $\inprod{h}{g}_{\mathcal{H}} = 0$ for any $g \in \mathcal{H}$ such
  that $g \equiv 0$ on $A$.
\end{cor}

\begin{proof}
  For any such $g$, $g+h \equiv 1$ on $A$, and so by Proposition
  \ref{T:inprod-magnitude},
  \[
  \inprod{h}{g}_{\mathcal{H}} = \inprod{h}{g+h}_{\mathcal{H}} -
  \inprod{h}{h}_{\mathcal{H}} = \abs{A} - \abs{A} = 0.
  \qedhere
  \]
\end{proof}

We consider next the role of measures which are not finitely
supported.  Let $M^{1/2}$ be the space of finite signed or complex
Borel measures $\mu \in M(X)$ such that $\int \sqrt{d(x,y)} \
d\abs{\mu}(y) < \infty$ for some (hence for any) $x \in X$.  It is
easy to verify that $h \mapsto \int h \ d\mu$ defines a bounded linear
functional on $C^{1/2}$, and therefore on $\mathcal{H}$.  Thus
measures in $M^{1/2}$ define elements of $\mathcal{W}$, and in
particular any $\mu \in M(A)$ defines an element of $\mathcal{W}_A$;
the map $Z$ acts on $\mu \in M^{1/2}$ according to \eqref{E:Z}. It is
not clear, however, whether distinct measures in $M^{1/2}$ necessarily
give rise to distinct elements of $\mathcal{W}$.  Two equivalent
formulations of this issue are that the form
$\inprod{\cdot}{\cdot}_{\mathcal{W}}$ may be degenerate on $M^{1/2}$,
and that $Z$ may not be injective on $M^{1/2}$.  This is why only
finitely supported measures were used in defining $\mathcal{W}$ in the
previous section.  However, this does imply that in
\eqref{E:magnitude-sup}, one can extend the supremum to all $\mu \in
M(A)$ such that $\norm{\mu}_{\mathcal{W}} \neq 0$; this leads to
another proof of the first part of \cite[Theorem 2.4]{Meckes}.

The above observation also lets us reproduce \cite[Theorem
2.3]{Meckes}, which shows that if $A$ possesses a \dfn{weight measure}
$\mu \in M(A)$---that is, $Z\mu (a) = \int e^{-d(a,b)} \ d\mu(b) = 1$
for each $a \in A$---then
\[
\abs{A} = \int (Z\mu) \ d\mu = \mu(A);
\]
in other words, the obvious analogue of the formula
\eqref{E:finite-magnitude} for the magnitude of a finite set does hold
in this case.  In \cite{Willerton-homog} weight measures are used to
define the magnitude of metric spaces which are not necessarily
positive definite (this will be discussed in more detail in Section
\ref{S:dimension} below); by \cite[Theorem 2.3]{Meckes} this coincides
with the present definition for any PDMS which possesses a weight
measure.  Although it is not clear how to define weightings or
magnitude for arbitrary metric spaces, it is clear that using weight
measures is insufficient since in general the weighting of a PDMS is
not given by a measure.  Nevertheless, it appears to be reasonable to
use this definition whenever a weight measure does exist.

We end this section by considering a quantity related to magnitude
which is in some ways better behaved.  For a compact (not necessarily
positive definite) metric space $A$, the \dfn{maximum diversity} of
$A$ is
\begin{equation} \label{E:diversity}
\abs{A}_+ = \sup_{\mu \in P(A)} \left( \int \int e^{-d(a,b)} \ d\mu(a)
  \ d\mu(b) \right)^{-1},
\end{equation}
where $P(A)$ denotes the space of Borel probability measures on $A$.
By renormalization, this is simply what one obtains by restricting the
supremum in \eqref{E:magnitude-sup} to positive measures; thus we
trivially have
\begin{equation} \label{E:diversity-magnitude}
 \abs{A}_+ \le \abs{A}
\end{equation}
for any compact PDMS $A$. The name stems from the following
interpretation of the quantity inside the supremum.  Suppose that the
points of a finite metric space $A$ represent all the species present
in an ecosystem, and $e^{-d(a,b)}\in (0,1]$ is viewed as the
similarity between species $a$ and $b$.  If $\mu \in P(A)$ describes the relative
abundances of the species, then
\begin{equation} \label{E:similarity}
\int \int e^{-d(a,b)} \ d\mu(a) \ d\mu(b)
\end{equation}
is the expected similarity of a pair of independently picked random
organisms. The reciprocal of \eqref{E:similarity} quantifies, in a
similarity-sensitive way, the diversity of the ecosystem; $\abs{A}_+$
is thus the maximum possible diversity of the given collection of
species when one considers all possible relative abundances. 

We remark that the reciprocal of \eqref{E:similarity} is just one of
an infinite family of diversity measures introduced in
\cite{Leinster-entropy,LeCo}. It is shown in \cite{Leinster-entropy}
that, under certain conditions, they all have the same maximum value.
See \cite{LeCo} for a thorough discussion of these diversity measures
in the context of theoretical ecology, and \cite{Leinster-entropy} for
a proof of a subtler relationship between maximum diversity and
magnitude.

Although maximum diversity lacks the category-theoretic motivation of
magnitude, it is in some ways more tractable than magnitude due to the
fact that it can be represented in terms of \emph{positive} measures.
One easy property which follows immediately from \eqref{E:diversity}
is that
\[
\abs{A}_+ \le \exp( \diam A),
\]
where $\diam A$ denotes the diameter of $A$. A subtler property is
that $\abs{A}_+$ is a continuous function of $A$ with respect to the
Gromov--Hausdorff topology (see \cite[Proposition 2.11]{Meckes}; this
result is stated for PDMSs but the proof applies to arbitrary compact
metric spaces).  Both of these properties fail in general for
magnitude, as witnessed by Example 2.2.8 in \cite{Leinster}, due to
Willerton, of a PDMS $A$ with six points such that $\lim_{t \to 0^+}
\abs{tA} = 6/5$.  Nevertheless, we will see in Section
\ref{S:capacity} below that magnitude and maximum diversity have a
deep relationship in Euclidean space. This will play a crucial role in
our investigation of the asymptotic growth of the magnitude function
in Section \ref{S:dimension}.

The maximum diversity of a compact PDMS can also be given a dual
characterization analogous to Theorem \ref{T:inf-magnitude}.

\begin{prop} \label{T:inf-diversity} Let $X$ be a PDMS and let $A
  \subseteq X$ be compact.  Then
  \begin{equation}\label{E:inf-capacity}
  \abs{A}_+ = \inf \Set{ \norm{h}_{\mathcal{H}}^2}{h \in
    \mathcal{H} \text{ and } h \ge 1 \text{ on } A}.
  \end{equation}
\end{prop}

\begin{proof}
  For simplicity of exposition we assume in this proof that $\F = \R$;
  the complex case then follows since $\norm{\Re h}_{\mathcal{H}} \le
  \norm{h}_{\mathcal{H}}$.

  By \eqref{E:diversity} and the duality between $\mathcal{W}$ and
  $\mathcal{H}$,
  \[
  \abs{A}_+^{-1/2} = \inf_{\mu \in P(A)} \norm{\mu}_{\mathcal{W}} =
  \inf_{\mu \in P(A)} \sup_{\substack{g \in \mathcal{H} \\
      \norm{g}_{\mathcal{H}}\le 1}} \int g \ d\mu.
  \]
  By a general minimax theorem (see e.g.\ \cite[Theorem 2.4.1]{AdHe}),
  this is equal to
  \[
  \sup_{\substack{g \in \mathcal{H} \\ \norm{g}_{\mathcal{H}}\le 1}}
  \inf_{\mu \in P(A)} \int g \ d\mu
  = \sup_{\substack{g \in \mathcal{H} \\ \norm{g}_{\mathcal{H}}\le 1}}
  \inf_{a \in A} g(a).
  \]
  Since $\mathcal{H}$ contains strictly positive functions, the
  supremum is unchanged if $g$ is assumed to be positive on $A$.
  Given a $g \in \mathcal{H}$ with $\inf_{a \in A} g(a) = c > 0$,
  define $h = \frac{1}{c} g$.  The mapping $g \mapsto h$ defines a
  bijection
  \[
  \Set{g \in \mathcal{H}}{\norm{g}_{\mathcal{H}} \le 1 \text{ and } g
    > 0 \text{ on } A}
  \to
  \Set{h \in \mathcal{H}}{h \ge 1 \text{ on } A},
  \]
  such that $\inf_{a \in A} g(a) = c = \norm{h}_{\mathcal{H}}^{-1}$.
  It follows that
  \[
  \abs{A}_+^{-1/2} = \sup \Set{\norm{h}_{\mathcal{H}}^{-1}}{h \in
    \mathcal{H}
    \text{ and } h \ge 1 \text{ on } A}.
  \qedhere
  \]
\end{proof}


\section{Magnitude in Euclidean space}
\label{S:Euclidean}

We now specialize to the case in which $X = \R^n$, equipped with the
standard inner product $\inprod{\cdot}{\cdot}$ and the associated norm
$\norm{\cdot}$ and metric $d$.  Some of the results of this section
generalize to certain other normed or quasinormed spaces (cf.\
Sections 3 and 4 of \cite{Meckes}), but here we restrict to the
Euclidean case, which is of most central interest and about which the
most can be said.

The first task is to observe that the weighting space $\mathcal{W}$
and the RKHS $\mathcal{H}$ turn out to be well-known Sobolev-type
spaces of distributions and functions, respectively.  Define $F:\R^n
\to \R$ by $F(x) := e^{-\norm{x}}$.  Then the Fourier transform of $F$
is
\begin{equation} \label{E:F-FT}
\widehat{F}(x) = \frac{1}{(2\pi)^{n/2}}
 \int F(y) e^{-i\inprod{x}{y}} \ dy 
 = \frac{n! \omega_n}{(2\pi)^{n/2}} \bigl(1 + \norm{x}^2 \bigr)^{-(n+1)/2},
\end{equation}
where $\omega_n = \pi^{n/2} /\Gamma\bigl(\frac{n}{2} + 1\bigr)$ is the
volume of the unit ball in $\R^n$ \cite[Theorem 1.14]{StWe}.  In this
setting the map $Z:FM \to C^{1/2}$ is the convolution operator $\mu
\mapsto F * \mu$, and therefore
\begin{equation} \label{E:Z-FT} 
  \widehat{Z\mu}(x) = (2\pi)^{n/2}
  \widehat{F}(x) \widehat{\mu}(x) = n! \omega_n \bigl(1 +
  \norm{x}^2 \bigr)^{-(n+1)/2} \widehat{\mu}(x).
\end{equation}

For $\alpha \in \R$, the \dfn{Bessel potential space} $H^\alpha =
H^\alpha(\R^n)$ is the Hilbert space of tempered distributions
\[
H^\alpha := \Set{ \varphi \in \mathcal{S}'(\R^n)}{\bigl(1 +
  \norm{\cdot}^2\bigr)^{\alpha/2} \widehat{\varphi} \in L^2(\R^n)}
\]
equipped with the norm
\[
\norm{\varphi}_{H^\alpha} := \sqrt{\int_{\R^n} \bigl(1 + \norm{x}^2\bigr)^\alpha
  \abs{\widehat{\varphi(x)}}^2 \ dx}.
\]
(See e.g.\ \cite[Section 7.9]{Hoermander1}; our normalization for the
Fourier transform---identified in \eqref{E:F-FT}---differs from that
of \cite{Hoermander1}, but the $H^\alpha$ norms agree.)  When $\alpha
\ge 0$, $H^\alpha \subseteq L^2$, and so $H^\alpha$ is actually a
space of functions.  For each $\alpha > 0$, the space $\mathcal{S}$ of
Schwartz functions is dense in $H^\alpha$, and $H^\alpha$ and
$H^{-\alpha}$ act as duals via the unique extension of the bilinear
pairing
\[
\inprod{f}{\varphi} = \varphi(f)
\]
for $f \in \mathcal{S} \subseteq H^\alpha$ and $\varphi \in H^{-\alpha}
\subseteq \mathcal{S}'$.

The following results show that when $X = \R^n$, the weighting space
$\mathcal{W}$ may be identified with the space of distributions
$H^{-(n+1)/2}$ and the RKHS $\mathcal{H}$ is the function space
$H^{(n+1)/2}$.  Thus, the potential function of a compact set $A
\subseteq \R^n$ is in fact a so-called Bessel potential.

\begin{prop}\label{T:W-H-s}
  The inclusion map $FM \hookrightarrow H^{-(n+1)/2}$ extends to a
  bijection $\mathcal{W} \to H^{-(n+1)/2}$ such that
  $\norm{\mu}_{\mathcal{W}} = \sqrt{n! \omega_n}
  \norm{\mu}_{H^{-(n+1)/2}}$ for each $\mu \in FM$.
\end{prop}

\begin{proof}
  For each $x \in \R^n$, $\bigl\lvert\widehat{\delta_x}\bigr\rvert \equiv
  (2\pi)^{-n/2}$.  Since
  \[
  \int_{\R^n} \bigl(1 + \norm{x}^2\bigr)^{-(n+1)/2} \ dx < \infty, 
  \]
  as can be seen by integrating in polar coordinates, it follows that
  $\delta_x \in H^{-(n+1)/2}$, and therefore that $FM \subseteq
  H^{-(n+1)/2}$.

  Now let $\mu = \sum_{j=1}^n c_j \delta_{x_j} \in FM$.  By the
  Fourier inversion theorem,
  \begin{align*}
    n! \omega_n \norm{\mu}_{H^{-(n+1)/2}}^2 
    &= (2\pi)^{n/2} \int_{\R^n} \widehat{F}
    \abs{\widehat{\mu}}^2 \\
    & = (2\pi)^{n/2} \int_{\R^n} \widehat{F}(y)
    \frac{1}{(2\pi)^n} \sum_{j,k=1}^n c_j \overline{c_k} e^{-i\inprod{y}{x_j - x_k}} \ dy \\
    &= \sum_{j,k=1}^n c_j \overline{c_k} \frac{1}{(2\pi)^{n/2}}
    \int_{\R^n} \widehat{F}(y) e^{i
      \inprod{y}{x_k-x_j}} \ dy \\
    &= \sum_{j,k=1}^n c_j \overline{c_k} F(x_k - x_j) =
    \norm{\mu}_{\mathcal{W}}^2.
  \end{align*}
  Thus the inclusion $FM \hookrightarrow H^{-(n+1)/2}$ extends to a
  map $\mathcal{W} \to H^{-(n+1)/2}$ with the stated identity between
  norms, and which is therefore injective.

  To show that this map is surjective, we need to show that $FM$ is
  dense in $H^{-(n+1)/2}$.  Equivalently, we need to show that any
  bounded linear functional on $H^{-(n+1)/2}$ which vanishes on $FM$
  is zero.  Each bounded linear functional on $H^{-(n+1)/2}$ is
  represented by some $f \in H^{(n+1)/2}$.  The evaluation of that
  linear functional on the point mass $\delta_x$ is $f(x)$; thus the
  vanishing of $f$ on $FM$ implies that $f = 0$.
\end{proof}

\begin{cor}\label{T:H-Hs}
  The map $Z:FM \to C^{1/2}$ extends to a bijection $Z:\mathcal{W} \to
  H^{(n+1)/2}$ such that $\norm{Z \mu}_{H^{(n+1)/2}} = \sqrt{n!
    \omega_n} \norm{\mu}_{\mathcal{W}}$ for each $\mu \in FM$.  Thus
  $\mathcal{H} = H^{(n+1)/2}$ and $\norm{h}_{H^{(n+1)/2}} = \sqrt{n!
    \omega_n} \norm{h}_{\mathcal{H}}$ for each $h \in H^{(n+1)/2}$.
\end{cor}

\begin{proof}
  This follows from Proposition \ref{T:W-H-s}, formula \eqref{E:Z-FT},
  and the fact that
  \[
  \varphi \mapsto \mathcal{F}^{-1} \Bigl(\bigl(1 + \norm{\cdot}^2\bigr)^{-(n+1)/2}
  \widehat{\varphi}\Bigr)
  \]
  is an isometry of $H^{-(n+1)/2}$ onto $H^{(n+1)/2}$, where
  $\mathcal{F}$ denotes the Fourier transform on $\mathcal{S}'$, which
  follows immediately from the definition of the spaces $H^\alpha$.
\end{proof}

Alternatively, in light of the comments in the previous section,
Corollary \ref{T:H-Hs} amounts to the known (but not easy to find
explicitly stated) fact that $H^{(n+1)/2}$ is an RKHS with reproducing
kernel $\frac{1}{n! \omega_n} e^{-\norm{x-y}}$.  (Note that $H^\alpha$
is only an RKHS when $\alpha > n/2$, and that there is no simple
explicit formula for the reproducing kernel for most values of
$\alpha$.)  Proposition \ref{T:W-H-s} then follows from the duality
between $H^{(n+1)/2}$ and $H^{-(n+1)/2}$.

If follows from Corollary \ref{T:H-Hs} that the general fact that
$\norm{h}_{C^{1/2}} \le \norm{h}_{\mathcal{H}}$ for $h \in
\mathcal{H}$ is, in the Euclidean setting, a special case of the
Sobolev embedding theorem.

\begin{cor}\label{T:magnitude-capacity}
  Let $A \subseteq \R^n$ be compact.  Then $A$ possesses a weighting
  in $H^{-(n+1)/2}$, and
  \[
  \abs{A} = \frac{1}{n! \omega_n} \inf \bigl\{ \norm{h}_{H^{(n+1)/2}}^2 \mid h
  \in H^{(n+1)/2} \text{ and } h \equiv 1 \text{ on } A \bigr\}.
  \]
\end{cor}

\begin{proof}
  Any Schwartz function $h$ with $h \equiv 1$ on $A$ lies in
  $H^{(n+1)/2}$.  Theorem \ref{T:inf-magnitude} and Corollary
  \ref{T:H-Hs} then imply the stated formula for $\abs{A}$ and the
  existence of a weighting for $A$, which by Proposition \ref{T:W-H-s}
  may be regarded as an element of $H^{-(n+1)/2}$.
\end{proof}

Corollary \ref{T:magnitude-capacity} contains the fact that each
compact subset of $\R^n$ has finite magnitude, first proved in
\cite[Proposition 3.5.3]{Leinster}. However, unlike the upper bound on
magnitude in \cite[Lemma 3.5.2]{Leinster}, from which Proposition
\cite[Proposition 3.5.3]{Leinster} was deduced, the infimum expression
in Corollary \ref{T:magnitude-capacity} is sharp, and implies some
new, sharper bounds. In the next result, the upper bound on $\abs{tA}$
improves \cite[Lemma 3.5.4]{Leinster}, which implies a similar upper
bound with an additional constant factor (depending on $A$).  The
lower bound is new.

\begin{thm}\label{T:magnitude-function}
  Let $A \subseteq \R^n$ be compact and $t \ge 1$.  Then 
  \[
  \frac{\abs{A}}{t} \le \abs{tA} \le t^n \abs{A}.
  \]
\end{thm}

\begin{proof}
  Let $h \in H^{(n+1)/2}$ be the potential function for $A$, and let
  $h_t (x) = h(x/t)$. By Corollary \ref{T:magnitude-capacity},
  \begin{align*}
  n! \omega_n \abs{tA} & \le \int_{\R^n} 
  \bigl(1 + \norm{x}^2\bigr)^{(n+1)/2} \abs{\widehat{h_t}(x)}^2 \ dx \\
  & = \int_{\R^n} \bigl(1 + \norm{x}^2\bigr)^{(n+1)/2} 
  \abs{t^{n} \widehat{h}(tx)}^2 \ dx \\
  & = t^{n} \int_{\R^n} \bigl(1 + t^{-2} \norm{y}^2\bigr)^{(n+1)/2} 
  \abs{\widehat{h}(y)}^2 \ dy \\
  & \le t^{n} \int_{\R^n} \bigl(1 + \norm{y}^2 \bigr)^{(n+1)/2}
  \abs{\widehat{h}(y)}^2 \ dy \\
  & = t^n n! \omega_n \abs{A}.
  \end{align*}

  Now let $g$ be the potential function for $tA$, and let $g^t (x) =
  g(tx)$.  By Corollary \ref{T:magnitude-capacity},
  \begin{align*}
  n! \omega_n \abs{A} & \le \int_{\R^n} 
  \bigl(1 + \norm{x}^2\bigr)^{(n+1)/2} \abs{\widehat{g^t}(x)}^2 \ dx \\
  & = \int_{\R^n} \bigl(1 + \norm{x}^2\bigr)^{(n+1)/2} 
  \abs{t^{-n} \widehat{g}(x/t)}^2 \ dx \\
  & = t^{-n} \int_{\R^n} \bigl(1 + t^2 \norm{y}^2\bigr)^{(n+1)/2} 
  \abs{\widehat{g}(y)}^2 \ dy \\
  & = t \int_{\R^n} \bigl(t^{-2} + \norm{y}^2 \bigr)^{(n+1)/2}
  \abs{\widehat{g}(y)}^2 \ dy \\
  & \le t \int_{\R^n} \bigl(1 + \norm{y}^2 \bigr)^{(n+1)/2}
  \abs{\widehat{g}(y)}^2 \ dy \\
  & = t n! \omega_n \abs{tA}.
  \qedhere
  \end{align*}
\end{proof}

It is an open question whether the magnitude function of a compact
subset of $\R^n$ (or more generally, a compact PDMS) must be
nondecreasing.  The lower bound in Theorem \ref{T:magnitude-function}
is one partial result in that direction; for another see the
discussion following Corollary \ref{T:magnitude-diversity} below.

\begin{cor} \label{T:magnitude-function-continuous} If $A \subseteq
  \R^n$ is compact, then the magnitude function of $A$ is continuous
  on $(0,\infty)$.
\end{cor}

\begin{proof}
  It suffices by rescaling to prove continuity of $\abs{tA}$ at $t =
  1$.  Theorem \ref{T:magnitude-function} immediately implies that
  $\lim_{t \to 1^+} \abs{tA} = \abs{A}$, and upon replacing $t$ with
  $1/t$ it also implies that $\lim_{t \to 1^-} \abs{tA} = \abs{A}$.
\end{proof}

As mentioned in Section \ref{S:weightings}, it was proved in
\cite{Meckes} that magnitude is lower semicontinuous on the class of
compact PDMSs; this implies that the magnitude function of a compact
space of negative type is lower semicontinuous.  Proposition 2.2.6 of
\cite{Leinster} implies that the magnitude function of a finite space
of negative type is even analytic.  It is unknown whether the
magnitude function of a compact space of negative type is continuous
for $t > 0$ in general, although \cite[Example 2.2.8]{Leinster},
mentioned above in Section \ref{S:potential}, shows that the magnitude
function can fail to be continuous at $0$ if we define $0A$ to be a
one-point metric space.

Corollary \ref{T:magnitude-capacity} can also be used to give an easy
proof of the following special case of \cite[Theorem 3.5.6]{Leinster}.
(The methods of the present paper can also be used to prove
\cite[Theorem 3.5.6]{Leinster}, which treats an arbitrary positive
definite finite dimensional normed space.)

\begin{prop} \label{T:magnitude-volume}
  If $A \subseteq \R^n$ is compact, then $\abs{A} \ge
  \frac{\vol(A)}{n! \omega_n}$.
\end{prop}

\begin{proof}
  Suppose $h \in H^{(n+1)/2}$ and $h \equiv 1$ on $A$.  Then by Parseval's
  identity,
  \[
  \norm{h}_{H^{(n+1)/2}}^2 \ge \bigl\lVert \widehat{h} \bigr\rVert_{L^2}^2 =
  \norm{h}_{L^2}^2 \ge \vol(A),
  \]
  and the result follows from Corollary \ref{T:magnitude-capacity}.
\end{proof}

We next show that the potential function of a compact subset of
Euclidean space satisfies a certain pseudodifferential equation, which
reduces to a partial differential equation in odd dimensions.  This
equation amounts to the Euler--Lagrange equation for the variational
problem described by Corollary \ref{T:magnitude-capacity}.

\begin{prop}\label{T:PDE}
  Let $A \subseteq \R^n$ be compact with potential
  function $h$. Then \( (I-\Delta)^{(n+1)/2} h = 0 \) in the
  distributional sense on $\R^n \setminus A$.
\end{prop}

\begin{proof}
  Let $g:\R^n \to \R$ be a smooth function with compact support
  contained in $\R^n \setminus A$. By Corollaries
  \ref{T:potential-orthogonal} and \ref{T:H-Hs},
  \begin{align*}
    \inprod{(I - \Delta)^{(n+1)/2} h}{g}
    &= \inprod{\mathcal{F} \bigl( (I - \Delta)^{(n+1)/2}
      h\bigr)}{\overline{\widehat{g}}} 
    = \inprod{\bigl(1 + \norm{\cdot}^2\bigr)^{(n+1)/2}
      \widehat{h}}{\overline{\widehat{g}}} \\
    &= \inprod{h}{g}_{H^{(n+1)/2}} = 0.
    \qedhere
  \end{align*}
  \end{proof}

\begin{cor} \label{T:smooth-potential}
  Let $A \subseteq \R^n$ be compact.  Then the potential function for $A$
  is $C^\infty$ on $\R^n \setminus A$.
\end{cor}

\begin{proof}
  The pseudodifferential operator $(I-\Delta)^{(n+1)/2}$ is elliptic,
  and the corollary follows from classical elliptic regularity
  theory; see e.g.\ \cite[Corollary 4.5]{Treves}.  For a more
  elementary treatment in the case that $n$ is odd (so that
  $(I-\Delta)^{(n+1)/2}$ is actually a differential operator), see
  the corollary to Theorem 8.12 in \cite{Rudin}.
\end{proof}

\begin{prop} \label{T:h-w}
  Let $A \subseteq \R^n$ be compact with potential function $h$.  Then
  the weighting $w$ of $A$ is the distribution
  \(
  w = \frac{1}{n! \omega_n} (I - \Delta)^{(n+1)/2} h.
  \)
\end{prop}

\begin{proof}
This follows from Corollary \ref{T:H-Hs}, \eqref{E:Z-FT}, and the fact
that $h = Zw$.
\end{proof}

As an application of the above results, we compute the potential
function, weighting, and magnitude of an interval $A = [0,\ell]
\subseteq \R$. Corollary \ref{T:smooth-potential} implies that the
potential function $h$ of $A$ is smooth on $(-\infty, 0)$ and on
$(\ell, \infty)$, and then Proposition \ref{T:PDE} shows that $h'' =
h$ (in the classical sense) on those intervals.  By definition, the
potential function $h$ is equal to $1$ on $A$, and is continuous by
Proposition \ref{T:W-bounded}.  Furthermore, since $h \in H^1(\R)$, it
decays to $0$ at $\pm \infty$ (see e.g.\ \cite[Corollary
7.9.4]{Hoermander1}). This boundary value problem has the unique
solution
\begin{equation}\label{E:h-interval}
h(x) = \begin{cases} e^x & \text{if } x < 0, \\
  1 & \text{if } 0 \le x \le \ell, \\
  e^{-x + \ell} & \text{if } x > \ell. \end{cases}
\end{equation}
By Proposition \ref{T:h-w}, the weighting of $A$ is the
distribution
\[
w = \frac{1}{2}(h - h'') = \frac{1}{2} \bigl(\lambda_A + \delta_0
+ \delta_\ell\bigr),
\]
where $\lambda_A$ denotes Lebesgue measure on the interval $A$.  Since
$w$ is a measure, Proposition \ref{T:inprod-magnitude} implies that
$\abs{A} = w(A) = 1 + \frac{\ell}{2}$.

The magnitude of an interval was first found in \cite[Theorem 7]{LW}
using approximation by finite sets, as justified by the results of
\cite{Meckes}; see also \cite[Theorem 3.2.2]{Leinster}.  The weight
measure of an interval was given, and proved to be a weight measure,
in \cite[Theorem 2]{Willerton-homog}, though it was not computed from
more basic data as above.  We note that Lemma 2.8 and Corollary 2.10
of \cite{Meckes} imply that the weighting of \emph{any} compact subset
of $\R$ is a measure, although numerical computations in
\cite{Willerton-heuristic} indicate this is unlikely to be true in
higher dimensions.


\section{Magnitude, diversity, and capacity} \label{S:capacity}

Experts in potential theory will have found several of the definitions
and results of Sections \ref{S:weightings} and \ref{S:potential} very
familiar, and recognized by Section \ref{S:Euclidean} that their
specializations to Euclidean space are rather classical.  In this
section, we make this connection explicit, and note that a deep
result about equivalence of capacities in Euclidean space implies an
important relationship between magnitude and maximum diversity, which
will be vital in our analysis of the growth of magnitude functions in
Section \ref{S:dimension} below.

Many definitions and results in potential theory have complicated
histories of successive generalizations.  The reader is referred to
\cite{AdHe} for original references.

For $\alpha > 0$, the \dfn{Bessel kernel} $G_\alpha : \R^n \to \R$
is defined as the function such that
\begin{equation} \label{E:Bessel-kernel}
\widehat{G_\alpha}(x) = (2\pi)^{-n/2} \bigl(1 +
\norm{x}^2\bigr)^{-\alpha / 2};
\end{equation}
see \cite[Section 1.2]{AdHe} for its basic properties.  
(Again, our normalization for the Fourier transform differs from that
of \cite{AdHe}, but the normalizations of $G_\alpha$ and of the norm
on $H^\alpha$---denoted by $L^{\alpha, 2}$ in \cite{AdHe}---are the
same.)  The \dfn{Bessel capacity} of order $\alpha$ of a compact set
$A \subseteq \R^n$ may defined in the following dual ways:
\begin{equation}\label{E:Bessel-capacity}
  \begin{split} C_\alpha(A) &:= \inf \Set{\norm{f}_{H^\alpha}^2}{f \in H^{\alpha}
      \text{ and } f \ge 1 \text{ on } A} \\
    & = \sup_{\mu \in P(A)} \norm{G_\alpha * \mu}_{L^2}^{-2}
    = \sup_{\mu \in P(A) \cap H^{-\alpha}} \norm{\mu}_{H^{-\alpha}}^{-2};
    \end{split}
\end{equation}
see Definition 2.2.6 and Theorem 2.2.7 of \cite{AdHe}.  (There is a
more general $L^p$ version of the Bessel capacity, a subject of
\emph{nonlinear} potential theory, which we need not consider here.)
By \eqref{E:diversity}, \eqref{E:F-FT}, and \eqref{E:Bessel-kernel}, it follows that for a
compact set $A \subseteq \R^n$,
\begin{equation} \label{E:diversity-capacity}
\abs{A}_+ = \frac{1}{n! \omega_n} C_{(n+1)/2}(A).
\end{equation}
Furthermore, the special case $\alpha = (n+1)/2$ of the equality in
\eqref{E:Bessel-capacity} is the same as the special case of
Proposition \ref{T:inf-diversity} for Euclidean space.  A corollary of
\eqref{E:diversity-capacity} is that if $A \subseteq \R$ is compact,
then $\abs{A} = \frac{1}{2} C_1(A)$.  This follows since $\abs{A} =
\abs{A}_+$ for compact $A \subseteq \R$, by \cite[Lemma 2.8]{Meckes}.

Another classical type of capacity\footnote{The author has been unable
  to find a standard name for this capacity in the literature.} of a
compact set $A \subseteq \R^n$ is
\begin{equation}\label{E:N-capacity}
  N_\alpha(A) := \inf \Set{\norm{f}_{H^\alpha}^2}{f \in \mathcal{S} \text{ and }
    f \equiv 1 \text{ on a neighborhood of } A};
\end{equation}
see \cite[Definition 2.7.1]{AdHe}. Note that $C_\alpha(A) \le
N_\alpha(A)$ trivially. Theorem \ref{T:inf-magnitude} indicates that
magnitude in $\R^n$ is closely related to the capacity $N_{(n+1)/2}$,
and in fact that for a compact set $A \subseteq \R^n$,
\begin{equation} \label{E:magnitude-capacity}
\abs{A} \le \frac{1}{n! \omega_n} N_{(n+1)/2}(A).
\end{equation}
There is also a dual formulation of $N_\alpha$, given in \cite[Theorem
2.7.2]{AdHe}, which closely parallels Theorem \ref{T:agreement}.

Although maximum diversity lacks the category-theoretic motivation
behind the definition of magnitude, it is in many ways easier to study
than magnitude due to its representation in terms of positive
measures. An analogous situation appears in potential theory:
$C_\alpha$ is simpler to analyze than $N_\alpha$ since it can also be
represented in terms of positive measures, but $N_\alpha$, whose dual
formulation requires signed measures or more general distributions,
arises naturally in certain applications, cf.\ \cite[p.\ 47]{AdHe}.

The following deep result gives the crucial relationship between the
two capacities $C_\alpha$ and $N_\alpha$ which allows $C_\alpha$ to be
used in some situations in which $N_\alpha$ appears more naturally.

\begin{prop}[{\cite[Corollary
    3.3.4]{AdHe}}] \label{T:equivalent-capacities} For each $\alpha >
  0$ and positive integer $n$ there exists a constant $\kappa(\alpha,
  n) > 0$ such that
  \[
  C_\alpha(A) \le N_\alpha(A) \le \kappa(\alpha, n) C_\alpha(A)
  \]
  for every compact set $A \subseteq \R^n$.
\end{prop}

This proposition essentially specializes to the following result about
magnitude.

\begin{cor}\label{T:magnitude-diversity}
  For each positive integer $n$ there exists a $ \kappa_n > 0$ such
  that for each compact set $A \subseteq \R^n$,
  \[
  \abs{A}_+ \le \abs{A} \le \kappa_n \abs{A}_+.
  \]
\end{cor}

\begin{proof}
  Setting $\alpha = (n+1)/2$, this follows immediately from
  \eqref{E:diversity-magnitude}, \eqref{E:diversity-capacity},
  \eqref{E:magnitude-capacity}, and Proposition
  \ref{T:equivalent-capacities}.
\end{proof}

As mentioned in Section \ref{S:Euclidean}, it is an open question
whether the magnitude function of a compact set $A \subseteq \R^n$ is
nondecreasing.  One consequence of Corollary
\ref{T:magnitude-diversity} is a partial result in this direction.  It
follows immediately from \eqref{E:diversity} that the \dfn{diversity
  function} $t \mapsto \abs{tA}_+$ is nondecreasing for $t > 0$, and
Corollary \ref{T:magnitude-diversity} implies that
\begin{equation} \label{E:magnitude-diversity-functions}
  \abs{tA}_+ \le \abs{tA} \le \kappa_n \abs{tA}_+;
\end{equation}
thus the magnitude function of $A$ is bounded above and below by
multiples of a nondecreasing function.  In particular, for $t \ge s$
we have that
\[
\abs{tA} \ge \abs{tA}_+ \ge \abs{sA}_+ \ge \kappa_n^{-1} \abs{sA},
\]
which roughly says that the magnitude function of $A$ never decreases
very much.  This complements the lower bound in Theorem
\ref{T:magnitude-function}, which yields a sharper estimate when $t$
is close to $s$ but a weaker estimate when $t \gg s$. The estimates in
\eqref{E:magnitude-diversity-functions} are most powerful when $t \to
\infty$, as will be exploited in Corollary \ref{T:Euclidean-dimension}
below.

At this point the reader may imagine that the theory of magnitude in
Euclidean space is already well-explored, albeit under a different
name, in the literature on capacities of sets.  In fact this is quite
far from the truth, and not simply---or even primarily---because of
the slight difference between the sets of functions appearing in
Corollary \ref{T:magnitude-capacity} and \eqref{E:N-capacity}.  In
typical applications, capacities are used to control the size of
``exceptional'' sets, and the principal interest is in sets of
capacity $0$.  Proposition \ref{T:equivalent-capacities} on the
equivalence of $C_\alpha$ and $N_\alpha$ is applied mainly via its
corollary that $C_\alpha(A) = 0$ if and only if $N_\alpha(A) = 0$ (see
\cite[Section 2.9]{AdHe} for discussion and references).  However, for
any $\alpha > n/2$, picking $\mu$ in the supremum in
\eqref{E:Bessel-capacity} to be a point mass shows that $N_\alpha(A)
\ge C_\alpha(A) > 0$ for each nonempty compact set $A \subseteq \R^n$.
From the point of view of traditional applications of capacities,
capacities of order $\alpha > n/2$---including magnitude and maximum
diversity, for which one can easily check that $\abs{A} \ge \abs{A}_+
\ge 1$ for any compact PDMS $A$---are thus somewhat pathological.

On the other hand, from the perspective that $\abs{A}$ is an effective
number of points of $A$, it is perfectly natural that $\abs{A} \ge 1$,
and the principal interest is in the magnitude of large sets,
particularly in the growth of the magnitude function of $A$ at
infinity.  It is thus quite interesting that Proposition
\ref{T:equivalent-capacities}, which was motivated by applications
involving sets of capacity $0$, turns out to be a vital ingredient of
the proof of Corollary \ref{T:Euclidean-dimension} below, which is the
main result about the growth of magnitude functions in Euclidean
space.

\section{Dimensions} \label{S:dimension}

We now turn to the investigation of the asymptotic growth of the
magnitude function.  Suppose that $A$ is a compact metric space of
negative type.  The \dfn{upper magnitude dimension} of $A$ is
\[
\umagdim A := \limsup_{t\to \infty} \frac{\log \abs{tA}}{\log t}
\]
and the \dfn{lower magnitude dimension} of $A$ is
\[
\lmagdim A := \liminf_{t\to \infty} \frac{\log \abs{tA}}{\log t}.
\]
When $\umagdim A = \lmagdim A$, or equivalently $\lim_{t\to \infty}
\frac{\log \abs{tA}}{\log t}$ exists, the \dfn{magnitude dimension}
$\magdim A$ is equal to this limit.  

Magnitude dimensions of various subsets of Euclidean space were
investigated in
\cite{LW,Willerton-heuristic,Leinster,Willerton-homog}. For example,
precise asymptotics of the magnitude function---which in particular
yield the magnitude dimension---were found for line segments
\cite[Theorem 7]{LW}, the Cantor set \cite[Theorem 11]{LW}, and
spheres in $\R^n$ \cite[Theorem 13]{Willerton-homog}.  The magnitude
dimension of the Sierpinski gasket was approximated numerically in
\cite[Section 4]{Willerton-heuristic}. Theorem 3.5.5 of
\cite{Leinster} (which is sharpened by Theorem
\ref{T:magnitude-function} above) implies that subsets of $\R^n$ have
upper magnitude dimension of at most $n$; and \cite[Theorem
3.5.6]{Leinster}, reproved for Euclidean space as Proposition
\ref{T:magnitude-volume} above, implies that subsets of $\R^n$ with
positive volume have magnitude dimension equal to $n$.  Theorems 3.4.8
and 3.5.6 of \cite{Leinster} extend these last two facts to the
$\ell^1$ metric on $\R^n$, and \cite[Theorems 4.4 and 4.5]{Meckes}
extend them to all $\ell^p$ metrics on $\R^n$ for $1 \le p \le 2$, and
in modified form to some related quasinorms\footnote{The last
  paragraph of the published version of \cite{Meckes} misstates the
  consequences for magnitude dimension of those results when the
  quasinorm in question has homogeneity of a degree smaller than
  $1$. A corrected version of \cite{Meckes} has been posted to the
  arXiv.}.  In all these cases the magnitude dimension was found to
agree with classical notions of dimension.

The main result of this section, Corollary
\ref{T:Euclidean-dimension}, unifies and generalizes all the results
about magnitude dimension in Euclidean space by proving that it is
always equal to Minkowski dimension. Toward this goal, we next recall
the definition of Minkowski dimensions for arbitrary compact metric
spaces and prove a new characterization of them in terms of maximum
diversity.

Let $A$ be any compact metric space.  For $\eps > 0$, the \dfn{packing
  number} $M(A,\eps)$ is the maximum number of disjoint closed
$\eps$-balls in $A$, and the \dfn{covering number} $N(A,\eps)$ is the
minimum number of closed $\eps$-balls needed to cover $A$.  The
quantities $\log N(A,\eps)$ and $\log M(A, \eps)$ are called the
$\eps$-\dfn{entropy} and $\eps$-\dfn{capacity} of $A$, respectively.

The \dfn{upper Minkowski dimension} of $A$ is
\[
\uboxdim A := \limsup_{\eps \to 0^+} \frac{\log
  N(A,\eps)}{\log (1/\eps)}
= \limsup_{\eps \to 0^+} \frac{\log
  M(A,\eps)}{\log (1/\eps)}
\]
and the \dfn{lower Minkowski dimension} of $A$ is
\[
\lboxdim A := \liminf_{\eps \to 0^+} \frac{\log
  N(A,\eps)}{\log (1/\eps)}
= \liminf_{\eps \to 0^+} \frac{\log
  M(A,\eps)}{\log (1/\eps)}.
\]
It is an easy exercise to prove the equalities between the two
expressions for each of these dimensions. When $\uboxdim A = \lboxdim
A$, or equivalently $\lim_{\eps \to 0^+} \frac{\log N(A,\eps)}{\log
  (1/\eps)}$ exists, the \dfn{Minkowski dimension} $\boxdim A$ is
equal to this limit.

The \dfn{upper}, \dfn{lower}, and ordinary \dfn{diversity dimension}
of an arbitrary compact metric space $A$ are defined analogously to
magnitude dimensions, using the maximum diversity $\abs{tA}_+$ in
place of the magnitude $\abs{tA}$; we denote them by $\ldivdim$,
$\udivdim$, and $\divdim$ respectively.  Observe that by
\eqref{E:diversity-magnitude}, for any compact space $A$ of negative
type,
\begin{equation} \label{E:divdim-magdim}
\ldivdim A \le \lmagdim A \qquad \text{and} \qquad
\udivdim A \le \umagdim A.
\end{equation}

\begin{thm} \label{T:box-equals-diversity} 
  For any compact metric space $A$, $\ldivdim A = \lboxdim A$ and
  $\udivdim A = \uboxdim A$.  Consequently, $\divdim A$ is defined if
  and only if $\boxdim A$ is defined, and in that case $\divdim A =
  \boxdim A$.
\end{thm}

\begin{proof}
  We prove first that $\ldivdim A \le \lboxdim A$ and $\udivdim A \le
  \uboxdim A$.  Let $\eps > 0$, $t > 0$, and $\mu \in P(A)$ be given.
  Observe that for each $a \in A$,
  \begin{equation} \label{E:integral-mu}
  \int e^{-t d(a,b)} \ d\mu(b) \ge \int_{B(a,\eps)} e^{-t d(a,b)} \
  d\mu(b)
  \ge e^{-t\eps} \mu(B(a,\eps)).
  \end{equation}
  Therefore, by Jensen's inequality,
  \begin{align*}
    \left(\int \int e^{-t d(a,b)} \ d\mu(a) \ d\mu(b) \right)^{-1} 
    & \le e^{t\eps} \left( \int \mu(B(a,\eps)) \ d\mu(a)\right)^{-1}
    \le e^{t\eps} \int \frac{1}{\mu(B(a,\eps))} \ d\mu(a).
  \end{align*}
  Now let $N = N(A, \eps/2)$, and let $a_1, \dotsc, a_N \in A$ such
  that $A = \bigcup_{j=1}^N B(a_j, \eps/2)$.  Suppose for the moment
  that $\mu(B(a_j, \eps/2)) > 0$ for each $j$.  If $a \in B(a_j,\eps/2)$
  then $B(a_j, \eps/2) \subseteq B(a,\eps)$, and so
  \begin{align*}
  \int \frac{1}{\mu(B(a,\eps))} \ d\mu(a) 
  &\le \sum_{j=1}^N \int_{B(a_j,\eps/2)} \frac{1}{\mu(B(a,\eps))} \ d\mu(a)
  \le \sum_{j=1}^N \int_{B(a_j,\eps/2)} \frac{1}{\mu(B(a_j,\eps/2))} \ d\mu(a)
  \\
  &= \sum_{j=1}^N \frac{\mu(B(a_j,\eps/2))}{\mu(B(a_j,\eps/2))} = N.
  \end{align*}
  If $\mu(B(a_j, \eps/2) = 0$ for some $j$, the sums above should be
  restricted to those $j$ for which $\mu(B(a_j,\eps/2)) > 0$, and one
  still obtains the upper bound of $N$.  Altogether, we have that
  \[
  \abs{tA}_+ = \sup_{\mu \in P(A)} \left(\int \int e^{-t d(a,b)} \
    d\mu(a) \ d\mu(b) \right)^{-1} 
  \le e^{t\eps} N(A,\eps/2).
  \]
  Setting $\eps = 2/t$, this suffices to prove that $\ldivdim A \le
  \lboxdim A$ and $\udivdim A \le \uboxdim A$.

  We next prove that $\lboxdim A \le \ldivdim A$ and $\uboxdim A \le
  \udivdim A$. Let $\eps > 0$ and $t > 0$ be given.  Let $M = M(A,
  \eps)$, let $a_1, \dotsc, a_M$ be the centers of disjoint closed
  $\eps$-balls in $A$, and define
  \[
  \mu := \frac{1}{M}\sum_{j=1}^M \delta_{a_j} \in P(A).
  \]
  For each $a \in A$, there is at most one $a_j$ in $B(a,\eps)$, and
  so
  \[
  \int e^{-td(a,b)} d\mu(b) 
  = \frac{1}{M} \sum_{j=1}^M e^{-td(a,a_j)} \le \frac{1}{M} + e^{-t\eps}.
  \]
  It follows that
  \[
  \frac{1}{\abs{tA}_+} \le \frac{1}{M(A,\eps)} + e^{-t\eps}.
  \]

  Now define $\eps(t) := \frac{\log(2 \abs{tA}_+)}{t}$ for $t \ge 1$, so
  that \( M(A, \eps(t)) \le 2 \abs{tA}_+ \).  We will prove below that
  $\eps(t)$ is a continuous and strictly decreasing function of $t$;
  for now, assume this to be the case. If $\eps(t)$ is bounded from
  below by a positive constant, then $\ldivdim A = \infty$ and the
  desired inequalities hold trivially.  We may thus assume that
  $\eps(t) \to 0$ as $t \to \infty$, and so
  \[
  \frac{M(A, \eps(t))}{\log (1/\eps(t))}
  \le
  \frac{\log (2 \abs{tA}_+)}{\log (1/\eps(t))}
  =
  \frac{\log \abs{tA}_+}{\log t} \frac{\log t}{\log (1/\eps(t))} + o(1)
  \]
  as $t \to \infty$.  Now
  \[
  \frac{\log(1/\eps(t))}{\log t} = 1 - \frac{\log \log (2\abs{tA}_+)}{\log t},
  \]
  and so if $t_n \to \infty$ such that $\frac{\log \abs{t_n A}_+}{\log
      t_n}$ is bounded above, then $\frac{\log(1/\eps(t_n))}{\log t_n}
    \to 1$ and thus
  \[
  \frac{M(A, \eps(t_n))}{\log (1/\eps(t_n))}
  \le
  \frac{\log \abs{t_n A}_+}{\log t_n} \bigl(1 + o(1)\bigr)
  \]
  as $n \to \infty$.  If $\ldivdim A < \infty$ then $\frac{\log
    \abs{t_n A}_+}{\log t_n}$ is bounded above for some sequence $t_n
  \to \infty$, and so
  \[
  \lboxdim A \le \liminf_{n \to \infty} \frac{M(A, \eps(t_n))}{\log
    (1/\eps(t_n))} 
  \le \liminf_{n \to \infty} \frac{\log \abs{t_n A}_+}{\log t_n}
  = \ldivdim A.
  \]

  If $\udivdim A < \infty$, then $\frac{\log \abs{tA}_+}{\log t}$ is
  bounded above for all $t \ge 1$, and since $\eps:[1,\infty) \to (0,
  \eps(1)]$ is bijective, we similarly obtain that $\uboxdim A <
  \udivdim A$.

  It remains to show that $\eps(t)$ is continuous and strictly
  decreasing on $(1,\infty)$. The continuity follows from
  \cite[Proposition 2.11]{Meckes}.  By definition,
  \[
  \abs{tA}_+^{1/t} = \sup_{\mu \in P(A)} \norm{e^{-d(\cdot, \cdot)}}_{L^t(\mu \otimes
    \mu)}^{-1}.
  \]
  By either H\"older's inequality or Jensen's inequality, for each
  fixed $\mu \in P(A)$, $\norm{e^{-d(\cdot, \cdot)}}_{L^t(\mu \otimes
    \mu)}$ is a nondecreasing function of $t$.  Thus
  $\abs{tA}_+^{1/t}$ is the supremum of a family of nonincreasing
  functions of $t$, hence nonincreasing, and so $2^{1/t}
  \abs{tA}_+^{1/t}$ is a strictly decreasing function of $t$.  The
  claim follows since the logarithm is a strictly increasing function.
\end{proof}

In the setting of Euclidean space, Theorem
\ref{T:box-equals-diversity} amounts to a characterization of
Minkowski dimension in terms of Bessel capacities $C_{(n+1)/2}$.
There are well-known relationships between the Hausdorff dimension of
sets in $\R^n$ and Bessel capacities $C_\alpha$ for $\alpha \le n/2$
(see \cite[Section 5.1]{AdHe}); this connection between Bessel
capacities and Minkowski dimension appears to be new.

Theorem \ref{T:box-equals-diversity} yields the following comparison
between magnitude dimension and Minkowski dimension for general spaces
of negative type.

\begin{cor} \label{T:box-le-mag}
  Let $A$ be a compact metric space of negative type.  Then $\lboxdim
  A \le \lmagdim A$ and $\uboxdim A \le \umagdim A$.
\end{cor}

\begin{proof}
  This follows immediately from Theorem \ref{T:box-equals-diversity}
  and \eqref{E:divdim-magdim}.
\end{proof}

For certain classes of compact PDMSs, magnitude is always equal to
maximum diversity, for example for subsets of $\R$, ultrametric spaces
(see the next section for the definition), or homogeneous PDMSs; see
\cite[Lemma 2.8]{Meckes}.  For spaces of negative type with this
property, Theorem \ref{T:box-equals-diversity} implies that magnitude
dimensions and Minkowski dimensions agree.  For example, we have the
following consequence for ultrametric spaces.  (Subsets of $\R$ and
homogeneous spaces are covered by Corollary
\ref{T:Euclidean-dimension} and Proposition \ref{T:homog-dimension}
below.)

\begin{cor} \label{T:ultrametric-dimension} If $A$ is a compact
  ultrametric space, then $\umagdim A = \uboxdim A$ and $\lmagdim A =
  \lboxdim A$.  Consequently, $\magdim A$ is defined if and only if
  $\boxdim A$ is defined, and in that case $\magdim A = \boxdim A$.
\end{cor}

When combined with Proposition \ref{T:equivalent-capacities}, Theorem
\ref{T:box-equals-diversity} has the deeper consequence is that
magnitude dimensions and Minkowski dimensions always agree in
Euclidean space.

\begin{cor} \label{T:Euclidean-dimension} If $A \subseteq \R^n$ is
  compact, then $\umagdim A = \uboxdim A$ and $\lmagdim A = \lboxdim
  A$.  Consequently, $\magdim A$ is defined if and only if $\boxdim A$
  is defined, and in that case $\magdim A = \boxdim A$.
\end{cor}

\begin{proof}
  For a compact set $A \subseteq \R^n$,
  \eqref{E:magnitude-diversity-functions} implies that $\umagdim A =
  \udivdim A$ and $\lmagdim A = \ldivdim A$, and the result follows
  from Theorem \ref{T:box-equals-diversity}.
\end{proof}

It remains an open question whether magnitude dimension is equal to
Minkowski dimension for each compact metric space of negative type.

As mentioned earlier, another approach to defining magnitude for
infinite spaces is to define a \dfn{weight measure} for a compact
metric space $A$ to be a signed measure $\mu \in M(A)$ such that, for
each $a \in A$,
\[
\int e^{-d(a,b)} \ d\mu(b) = 1;
\]
and then define $\abs{A} := \mu(A)$ whenever $\mu$ is a weight measure
for $A$.  This clearly extends the original definition
\ref{D:finite-magnitude} for the magnitude of finite spaces, and, as
discussed in Section \ref{S:potential}, can be proved to coincide with
Definition \ref{D:compact-magnitude} whenever $A$ is a compact PDMS
which possesses a weight measure.  If $tA$ possesses a weight measure
for each $t > 0$, then we define the magnitude function and upper,
lower, and ordinary magnitude dimensions of $A$ as before.

This definition of magnitude is useful in particular when $A$ is a
compact homogeneous metric space (i.e., the isometry group acts
transitively on the points of $A$).  In this case there exists a
unique isometry-invariant probability measure $\mu \in P(A)$ (see,
e.g., \cite[Theorem 1.3]{MiSc}), which is also isometry-invariant on
$tA$ for each $t > 0$.  Theorem 1 of \cite{Willerton-homog} then shows
that an appropriate scalar multiple of $\mu$ is a weight measure for
$tA$, and for each $a \in A$,
\begin{equation} \label{E:homog-magnitude}
  \abs{tA} = \left(\int e^{-td(a,b)} \ d\mu(b)\right)^{-1}.
\end{equation}

Using this definition of magnitude, precise asymptotics for the
magnitude function of a compact homogeneous Riemannian manifold were
found in \cite[Theorem 11]{Willerton-homog}; these imply that for such
manifolds, the magnitude dimension equals the usual dimension.
Similar arguments as in the proof of Theorem
\ref{T:box-equals-diversity} generalize this fact---with Minkowski
dimension in place of the dimension of a manifold---to arbitrary
compact homogeneous metric spaces.  The existence of an invariant
weight measure takes the place in this setting of the equivalence of
capacities from Proposition \ref{T:equivalent-capacities}.

\begin{prop} \label{T:homog-dimension} If $A$ is a compact homogeneous
  metric space then $\umagdim A = \uboxdim A$ and $\lmagdim A =
  \lboxdim A$.  Consequently, $\magdim A$ is defined if and only if
  $\boxdim A$ is defined, and in that case $\magdim A = \boxdim A$.
\end{prop}

\begin{proof}
  Let $\mu$ be the unique isometry-invariant probability measure on
  $A$. Let $N = N(A,\eps)$, and let $a_1, \dotsc, a_N \in A$ such that
  $A = \bigcup_{j=1}^N B(a_j, \eps)$. Then for each $a \in A$,
  \[
  1 = \mu(A) \le \sum_{j=1}^N \mu(B(a_j,\eps)) = N \mu(B(a,\eps)).
  \]
  Together with \eqref{E:integral-mu} and \eqref{E:homog-magnitude},
  this implies that
  \[
  \abs{tA} \le e^{t\eps} N(A, \eps).
  \]
  Setting $\eps = 1/t$, this suffices to prove that $\lmagdim A \le
  \lboxdim A$ and $\umagdim A \le \uboxdim A$.
  
  Similarly, if $M = M(A,\eps)$ and $a_1, \dotsc, a_M \in A$ are the
  centers of disjoint balls of radius $\eps$, then for each $a \in A$,
  \[
  1 = \mu(A) \ge \sum_{j=1}^M \mu(B(a_j,\eps)) = M \mu(B(a,\eps)).
  \]
  Therefore,
  \begin{align*}
  \int e^{-t d(a,b)} \ d\mu(b) &= \int_{B(a,\eps)} e^{-td(a,b)} \ d\mu(b)
  + \int_{A \setminus B(a,\eps)} e^{-td(a,b)} \ d\mu(b) \\
  & \le \mu(B(a,\eps)) + e^{-t\eps}
  \le \frac{1}{M(A,\eps)} + e^{-t\eps}.
  \end{align*}
  Together with \eqref{E:homog-magnitude}, this implies that
  \[
  \frac{1}{\abs{tA}} \le \frac{1}{M(A,\eps)} + e^{-t\eps},
  \]
  and the proof is completed as in the second half of the proof of
  Theorem \ref{T:box-equals-diversity}.
\end{proof}


\section{Afterword: ultramagnitude of ultrametric
  spaces} \label{S:ultrametric}

As discussed in the introduction, the magnitude of a finite metric
space is a special case of the more general notion of the magnitude of
a finite enriched category, presented in \cite[Section 1]{Leinster}.
Besides the cases of ordinary categories (for which magnitude is known
as Euler characteristic, and is related to more classical invariants
of that name) and of metric spaces, the magnitude of enriched
categories has mostly not yet been very fully explored (see
\cite{Leinster-cafe-enriched} for a discussion).  In this section, we
work out another special case, that of ultrametric spaces.  This leads
to an extremely simple notion of the size of an ultrametric space,
whose theory is similar to, but drastically simpler than, the theory
of magnitude of metric spaces.  In particular, the notions of
packings, coverings, and Minkowski dimensions, which played central
roles in the previous section, come immediately out of this theory.

An \dfn{ultrametric space} is a metric space $(A, d)$ which satisfies
the strengthened triangle inequality
\[
d(a,c) \le \max \bigl\{ d(a,b), d(b,c) \bigr\}
\]
for each $a,b,c \in A$.  That is, one obtains the definition of an
ultrametric space by replacing the binary operation $+$ in the
definition of a metric space with the binary operation $\max$.  Of
course, ultrametric spaces are in particular metric spaces, and are
even always positive definite (see \cite[Proposition 2.4.18]{Leinster}
or \cite[Theorem 3.6]{Meckes}).  Thus, one can speak of the magnitude
of a compact ultrametric space, as we have done in Corollary
\ref{T:ultrametric-dimension} above.  However, one obtains a different
notion if one appropriately substitutes the operation of $\max$ for
the operation $+$, not in Definition \ref{D:finite-magnitude} itself,
but in the category-theoretic considerations which motivate that
definition.  To distinguish this new notion from the magnitude of $A$
when thought of simply as a metric space, we will call it
``ultramagnitude''.

Definition \ref{D:finite-magnitude} of the magnitude of a metric space
$A$ is built around the matrix $\zeta(a,b) = e^{-d(a,b)}$, the
motivation for which was not explained in this paper.  We will now
explain just the part of its motivation which needs to be modified for
ultrametric spaces. The reader is referred to \cite[Section
1]{Leinster} and the references therein for the full definition of
magnitude of an enriched category and its category-theoretic
background, or to \cite[Section 1.1]{LW} for a brief summary.  Here we
will bring up only the essential minimum, noting for experts that
whereas a metric space is a category enriched over the monoidal
category $\bigl( ([0,\infty), \ge), +, 0 \bigr)$, an ultrametric space
is a category enriched over $\bigl( ([0,\infty), \ge), \max,0 \bigr)$.

When specialized to metric spaces, Leinster's definition of the
magnitude of an enriched category calls for a function
$\Phi:[0,\infty) \to \R$ such that
\[
\Phi(x+y) = \Phi(x) \Phi(y)
\]
for each $x,y \in [0,\infty)$. One then defines $\zeta(a,b) =
\Phi(d(a,b))$.  Here the domain $[0,\infty)$ is the set of possible
distances (the objects of the enriching category) and the operation
$+$ in $x+y$ is the same operation appearing in the triangle
inequality (the tensor product in the enriching category).  If $\Phi$
is to be Lebesgue measurable, we must have $\Phi(x) = \alpha^x$ for
some $\alpha \ge 0$ (see \cite{Frechet}); the choice of
$\alpha=e^{-1}$ is the arbitrary choice of scale in Definition
\ref{D:finite-magnitude} which is addressed by considering magnitude
functions.

To adapt Definition \ref{D:finite-magnitude} to ultrametric spaces, we
instead need a function $\Psi:[0,\infty) \to \R$ such that
\[
\Psi\bigl(\max \{x,y\}\bigr) = \Psi(x) \Psi(y)
\]
for each $x, y \in [0,\infty)$. This implies that $\Psi$ must be the
indicator function of some interval $[0,\beta]$ or $[0,\beta)$.  We
choose $\beta=1$, which, like the choice $\alpha = e^{-1}$ above,
amounts to a convenient but arbitrary choice of scale.  We furthermore
pick $\Psi$ to be the indicator function of $[0,1]$, which amounts to
a choice to work with closed balls as opposed to open balls.  The
general definition of the magnitude of a finite enriched category then
specializes to ultrametric spaces in the following way.

\begin{defn} \label{D:ultramagnitude} Given a finite ultrametric space
  $(A, d)$, define the matrix $\xi \in \R^{A \times A}$ by
\[
\xi(a,b) := \begin{cases} 1 & \text{if } d(a,b) \le 1, \\
  0 & \text{if } d(a,b) > 1.
  \end{cases}
\]
A vector $w \in \R^A$ is an \dfn{ultraweighting} for $A$ if for each $a \in
A$,
\[
  (\xi w)(a) = \sum_{b \in A} \xi(a,b) w(b) = 1.
\]
If $A$ possesses an ultraweighting $w$, then the \dfn{ultramagnitude}
of $A$ is
\[
  \abs{A}_U := \sum_{a \in A} w(a).
\]
\end{defn}

Observe that the matrix $\xi$ used in Definition
\ref{D:ultramagnitude} is a discretization of the matrix $\zeta$ from
Definition \ref{D:finite-magnitude}.  This suggests that magnitude
should reflect finer information than ultramagnitude.

The entire theory of the ultramagnitude of finite ultrametric spaces
can be summed up in the following result.

\begin{prop} \label{T:ultramagnitude} Let $A$ be a finite ultrametric
  space.  Then $\abs{A}_U = N(A,1) = M(A,1)$.
\end{prop}

\begin{proof}
  It is an easy exercise to check that the closed balls of radius $1$
  in an ultrametric space $A$ form a partition of $A$, which must thus
  consist of $N(A,1) = M(A,1)$ distinct balls.  An ultraweighting for
  $A$ is given by
  \[
  w(a) = \bigl( \# B(a,1) \bigr)^{-1},
  \]
  where $\#$ denotes cardinality.  Therefore,
  \[
  \abs{A}_U = \sum_{a\in A} w(a) = \sum_{\text{distinct }B(a,1)} 1 =
  N(A,1).
  \qedhere
  \]
\end{proof}

Given Proposition \ref{T:ultramagnitude}, it is simple to extend the
definition of ultramagnitude to compact ultrametric spaces in a
natural way: we simply let $\abs{A}_U = N(A,1)$.  (One can also arrive
at this definition, with rather more effort, by appropriately
modifying the approach of either \cite{Meckes} or Section
\ref{S:weightings} above, but we will not pursue this here.)  For
$t>0$, it follows that $\abs{tA}_U = N(A,1/t) = M(A,1/t)$.  Thus the
\dfn{ultramagnitude function} $t \mapsto \abs{tA}_U$ contains
precisely the same information as the $\eps$-entropy or
$\eps$-capacity of a compact ultrametric space, and it follows
trivially that ``ultramagnitude dimension'' is equal to Minkowski
dimension.


\end{document}